\documentclass[12pt]{article}

\usepackage[notref,notcite]{showkeys}
\usepackage[colorlinks,linkcolor=blue,citecolor=blue]{hyperref}
\usepackage{amsthm}
\usepackage{multirow}
\usepackage{marginnote}
\usepackage{a4wide}
\usepackage{amssymb}
\usepackage{amsfonts}
\usepackage{amsmath}
\usepackage{mathrsfs}
\usepackage{tikz}
\usetikzlibrary{arrows,matrix}
\usetikzlibrary{positioning}
\usepackage{mdframed} 
\usepackage{lipsum} 
\usepackage{extarrows} 
\usepackage{color,enumerate}
\usepackage{tocloft} 

\input xy
\xyoption{arrow} \xyoption{matrix}

\def\S{\mathcal{S}}
\def\Z{\mathbb{Z}}
\def\C{\mathbb{C}}
\def\N{\mathbb{N}}
\def\E{\tilde{E}}
\def\F{\tilde{F}}
\def\cN{\mathcal{N}}

\numberwithin{equation}{section} \theoremstyle{definition}

\date{}

\newtheorem{proposition}{Proposition}[section]
\newtheorem{theorem}[proposition]{Theorem}
\newtheorem{lemma}[proposition]{Lemma}

\newtheorem{corollary}[proposition]{Corollary}

\allowdisplaybreaks
\begin{document}

\author{Yan-an Cai, Yongsheng Cheng, Genqiang Liu}
\title{Simple weight modules over the quantum Schr\"{o}dinger algebra}

\maketitle

\begin{abstract}
In the present paper,  using the technique of localization, we  determine the center of the quantum Schr\"{o}dinger algebra $\S_q$  and classify simple  modules with finite-dimensional weight spaces over  $\S_q$, when $q$ is not a root of unity.
It turns out that there are four classes of such modules: dense $U_q(\mathfrak{sl}_2)$-modules, highest weight modules, lowest weight modules, and twisted modules of highest weight modules.

\end{abstract}

\section{Introduction}

In this paper, we denote by $\mathbb{Z}$, $\mathbb{Z}_+$, $\mathbb{N}$,
$\mathbb{C}$ and $\mathbb{C}^*$ the sets of  all integers, nonnegative integers,
positive integers, complex numbers, and nonzero complex numbers, respectively. Let $q$ be a nonzero complex number which is not a root of unity. For $n,i\in \mathbb{Z}$, denote
$[n]_q=\frac{q^n-q^{-n}}{q-q^{-1}}$, $\binom{n}{i}_q=\frac{[n]_q[n-1]_q\cdots[n-i+1]_q}{[i]_q[i-1]_q\cdots [1]_q}$. For an associative algebra $A$, we use $Z(A)$ to denote its center.

The representations of quantum groups have attracted extensive attention of many mathematicians and physicists. However most of the
research is related to the quantum groups of simple Lie algebras. In the present paper, we study the representations of the quantum group corresponding to a non-semisimple Lie algebra which is called the Schr\"{o}dinger Lie algebra. In the $(1+1)$-dimensional space, the Schr\"{o}dinger Lie algebra $\S$ is the semidirect product of $\mathfrak{sl}_2$ and the three-dimensional Heisenberg Lie algebra.
 It can describe symmetries of the free particle Schr\"{o}dinger equation, see \cite{DD1}. The representation theory of the Schr\"{o}dinger algebra  has been studied by  many authors.
 A classification of the simple
highest weight representations of the Schr\"{o}dinger algebra were given in \cite{DD1}. All simple weight modules with finite dimensional weight spaces were
classified in \cite{D}. The simple weight modules  of conformal Galilei algebra which generalized
Schr\"{o}dinger algebra in $l$-spatial dimension were studied in \cite{LMZ}. In \cite{ZC}, the authors
studied the Whittaker modules over  $\S$, simple
Whittaker modules and related Whittaker vectors were determined. Quasi-Whittaker modules over $\S$ were defined and classified in \cite{CCS}.

In 1996, in order to research the $q$-deformed heat equations,  a $q$-deformation of the universal enveloping algebra of the Schr\"{o}dinger Lie algebra was  introduced by Dobrev et al. , see \cite{DD}.  It is an associative  algebra over $\mathbb{C}$ generated by
$P_t,P_x,G, K_1, D, m$ subject to the following nontrivial relations:
\begin{align}
P_tG-qGP_t&=P_x,  &[P_x,K_1]&=Gq^{-D},  \, & [D,G]&=G,\\
 [D,P_x]&=-P_x,  &  [D,P_t] &=-2P_t,  \,& [D,K_1]&=2K_1,
\\
[P_t,K_1] &= \frac{q^D-q^{-D}}{q-q^{-1}},\, & P_xG-q^{-1} GP_x&=m, \, & P_tP_x-q^{-1}P_xP_t&=0.
\end{align}
If we denote $$K^{\pm 1}=q^{\pm D}, E=P_t, F=-K_1, Y=G, X=P_x, C=-m,$$ and replace $q$ with $q^{-1}$, then they satisfy the following relations:

\begin{align}
KEK^{-1}&=q^2 E,  &KFK^{-1}&=q^{-2}F,  \, & [E,F]&=\frac{K-K^{-1}}{q-q^{-1}}\label{R1},\\
 KXK^{-1} &=qX,  &  KYK^{-1} &=q^{-1}Y,  \,& qYX-XY&=C, \label{R2}
\\
EX &= qXE,\, & EY &= X+ q^{-1}YE, \, & [C,\mathcal{S}_q]&=0,\label{R3}\\
FX&= YK^{-1}+XF,\  &FY&=YF. \label{R4}
\end{align}

 Let $\S_q$ be the associative  algebra  over $\mathbb{C}$ generated by the elements $C, E, F,  K,  K^{-1}, X$ and $ Y$ subject to the defining relations
(\ref{R1})-(\ref{R4}).  We call $\S_q$ the \emph{quantum Schr\"{o}dinger algebra}.

For any $z\in \mathbb{C}$, the quotient algebra $\S_q/(C-z)\S_q$ is a  quantized symplectic oscillator algebras of rank one, see \cite{GK}. In particular, $\overline{\S}_q:=\S_q/C\S_q$ is the smash product of
the quantum plane $\mathbb{C}_q[X,Y]$ and $U_q(\mathfrak{sl}_2)$. We call $\overline{\S}_q$ the \emph{centerless quantum Schr\"{o}dinger algebra}. The subalgebra of $\overline{\S}_q$ generated by $E,K,K^{-1},X$ and $Y$ is the quantum spatial ageing algebra defined in \cite{BL}.


An $\S_q$-module $V$  is  called a weight module if $K$ acts diagonally on $V$, i.e.,
$$V=\oplus_{\lambda\in \mathbb{C^*}}V_\lambda,$$ where $V_\lambda=\{v\in V \mid K v=\lambda v\}$.  For $\omega\in \mathbb{C}^*$, denote $V(\omega)=\oplus_{i\in\mathbb{Z}}V_{\omega q^i}$.
If $V$ is simple, then $V= V(\omega)$ for some $\omega$. For a weight module $V$, let $\mathrm{supp}(V)=\{\lambda\in\C^*|V_\lambda\neq0\}$.

The goals of this paper are to determine the centers of $\overline{\S}_q$ and $\S_q$, and to classify all simple weight $\S_q$-modules with finite dimensional weight spaces.

For a simple weight $\S_q$-module $V$ with finite dimensional weight spaces, if $XV=YV=0$, then $M$ is a simple  $U_q(\mathfrak{sl}_2)$-module. All the simple $U_q(\mathfrak{sl}_2)$-modules were classified in \cite{B}.
Since highest (lowest) weight modules have been classified in \cite{DD}, it remains to classify
 those simple weight modules on which either $X$ or $Y$ or both act nonzero and, furthermore, which have neither
a highest nor a lowest weight. We denote the class of such modules by $\cN$.

The paper is organized as follows. In section 2, we will determine the center for the algebras $\overline{\S}_q$ and $\S_q$. In
section 3, some basic results for our discussions on weight modules will be given.  We will give details on twisting functors in section 4. Finally, in
section 5, we classify simple weight modules in $\cN$.

\section{The center of the algebras $\overline{\S}_q$ and $\S_q$}

In this section, we will determine the center for the algebra $\overline{\S}_q$ and $\S_q$. Indeed, we will prove the following theorem.
\begin{theorem}\label{center}
\begin{enumerate}[(i)]
\item The center of the centerless quantum Schr\"{o}dinger algebra $\overline{\S}_q$ is trivial.
\item The center of the quantum Schr\"{o}dinger algebra $\S_q$ is  $Z(\S_q)=\C[C]$.
\end{enumerate}
\end{theorem}

Note that this is not so similar with the center of the enveloping algebra of the (centerless) Schr\"{o}dinger algebra, see \cite{DLMZ}.

In \cite{GK}, using the action of the center  on simple highest weight modules, Gan and Khare showed that for any nonzero complex number $z$, the  center of $\S_q/(C-z)\S_q$ is trivial.
However, their method is not applicable to the case that $z=0$. When $C$ acts trivially on a simple highest weight module $M$, we must have that  both $X$ and $Y$ act trivially on $M$, see Proposition 3.10 in
\cite{GK}.

%
%
\

Before proving Theorem \ref{center}, we will give the following useful formulas in the centerless quantum Schr\"{o}dinger algebra.
\begin{lemma}
The following equalities hold in the centerless quantum Schr\"{o}dinger algebra.
\begin{align*}
&\E X=X\E,\,\, \E Y=q^{-1}Y\E,\,\, \E K=q^{-1}K\E,\\
&\F X=X\F,\,\, \F Y=q Y\F,\,\, \F K=qK\F,\\
&\E\F^i=\F^i\E+(q^{-2i}-1)\F^{i-1}XYK,\\
&\F\E^i=\E^i\F+q^{-2}(q^{2i}-1)\E^{i-1}XYK,
\end{align*}where $\E=EY-qYE=X+(q^{-1}-q)YE$, $\F=FX-q^{-2}XF=YK^{-1}+(1-q^{-2})XF$ and $i\in \mathbb{Z}_+$.
\end{lemma}

\begin{proof}
Following the defining relations, we have
\begin{align*}
\E X&=(X+(q^{-1}-q)YE)X=X^2+(q^{-1}-q)YEX\\
&=X^2+(q^{-1}-q)qYXE=X^2+(q^{-1}-q)XYE\\
&=X(X+(q^{-1}-q)YE)=X\E,\\
\\
\E Y&=(X+(q^{-1}-q)YE)Y=XY+(q^{-1}-q)YEY\\
&=qYX+(q^{-1}-q)Y(X+q^{-1}YE)=q^{-1}Y(X+(q^{-1}-q)YE)\\
&=q^{-1}Y\E,\\
\\
\E K&=(X+(q^{-1}-q)YE)K=XK+(q^{-1}-q)YEK\\
&=q^{-1}KX+(q^{-1}-q)q^{-2}YKE=q^{-1}K(X+(q^{-1}-q)YE)\\
&=q^{-1}K\E,\\
\\
\F X&=(YK^{-1}+(1-q^{-2})XF)X=YK^{-1}X+(1-q^{-2})XFX\\
&=q^{-1}YXK^{-1}+(1-q^{-2})X(YK^{-1}+XF)=X(YK^{-1}+(1-q^{-2})XF)\\
&=X\F,\\
\\
\F Y&=(YK^{-1}+(1-q^{-2})XF)Y=YK^{-1}Y+(1-q^{-2})XFY\\
&=qY^2K^{-1}+(1-q^{-2})XFY=qY(YK^{-1}+(1-q^{-2})XF)\\
&=qY\F,\\
\\
\F K&=(YK^{-1}+(1-q^{-2})XF)K=Y+(1-q^{-2})q^2XKF\\
&=qKYK^{-1}+q(1-q^{-2})KXF=qK(YK^{-1}+(1-q^{-2})XF)\\
&=qK\F.
\end{align*}

We use induction on $i$ to prove the last two equalities. First, we have
\begin{align*}
E\F&=E(FX-q^{-2}XF)\\
&=(FE+\frac{K-K^{-1}}{q-q^{-1}})X-q^{-1}XEF\\
&=qFXE+\frac{KX-K^{-1}X}{q-q^{-1}}-q^{-1}X(FE+\frac{K-K^{-1}}{q-q^{-1}})\\
&=q(FX-q^{-2}XF)E+\frac{KX-K^{-1}X}{q-q^{-1}}-\frac{q^{-1}XK-q^{-1}XK^{-1}}{q-q^{-1}}\\
&=q\F E+\frac{1-q^{-2}}{q-q^{-1}}KX\\
&=q\F E+q^{-1}KX.
\end{align*}
Hence, we have
\begin{align*}
\E\F&=(X+(q^{-1}-q)YE)\F=X\F+(q^{-1}-q)YE\F\\
&=\F X+(q^{-1}-q)Y(q\F E+q^{-1}KX)\\
&=\F X+(q^{-1}-q)qY\F E+(q^{-2}-1)YKX\\
&=\F X+(q^{-1}-q)\F YE+(q^{-2}-1)XYK\\
&=\F\E+(q^{-2}-1)XYK.
\end{align*}
This means the last two equalities hold for $i=1$. Suppose they are true for $i$, then
\begin{align*}
\E\F^{i+1}&=\E\F^i\F=\F^i\E\F+(q^{-2i}-1)\F^{i-1}XYK\F\\
&=\F^i(\F\E+(q^{-2}-1)XYK)+(q^{-2i}-1)q^{-2}\F^iXYK\\
&=\F^{i+1}\E+(q^{-2(i+1)}-1)\F^iXYK,\\
\\
\F\E^{i+1}&=\F\E^i\E=\E^i\F\E+q^{-2}(q^{2i}-1)\E^{i-1}XYK\E\\
&=\E^i(\E\F+(1-q^{-2})XYK)+(q^{2i}-1)\E^iXYK\\
&=\E^{i+1}\F+q^{-2}(q^{2(i+1)}-1)\E^iXYK.\tag*{\qedhere}
\end{align*}
\end{proof}

We will use localization to determine the center of the centerless quantum Schr\"{o}dinger algebra. Since we have
\begin{align*}
EY^i&=q^{-i}Y^iE+[i]_qY^{i-1}X-\frac{q+q^2-q^{2-i}-q^{i+1}}{(1-q^2)(q-1)}CY^{i-2},\\
XY^i&=q^iY^iX-\frac{q^i-1}{q-1}CY^{i-1},
\end{align*}
the set $\{Y^i|i\in\Z_+\}$ is a left and right Ore subset of $\overline{\S}_q$. Similarly, for any $s\in\{E,F,X,Y,C\}$, the set $\{s^i|i\in\Z_+\}$ is a left and right Ore subset of $\overline{\S}_q$. Hence, we can consider the corresponding localization $\overline{\S}_q^{(s)}$. For $\overline{\S}_q^{(X,Y)}$ we have the following analogue to the Poincar\'{e}-Birkhoff-Witt theorem.
\begin{lemma}
The set $\{X^aY^bK^c\E^d\F^e|(a,b,c,d,e)\in\Z\times\Z\times\Z\times\Z_+\times\Z_+\}$ is a basis for $\overline{\S}_q^{(X,Y)}$.
\end{lemma}
\begin{proof}
By Poincar\'{e}-Birkhoff-Witt theorem, we know that $\{X^aY^bK^cE^dF^e|(a,b,c,d,e)\in\Z_+\times\Z_+\times\Z\times\Z_+\times\Z_+\}$ is a basis for $\overline{\S}_q$. By the definition of localization, $\{X^aY^bK^cE^dF^e|\\(a,b,c,d,e)\in\Z\times\Z\times\Z\times\Z_+\times\Z_+\}$ is a basis for $\overline{\S}_q^{(X,Y)}$, and hence $\{X^aY^bK^c\E^d\F^e|(a,b,c,d,e)\\\in\Z\times\Z\times\Z\times\Z_+\times\Z_+\}$ spans  $\overline{\S}_q^{(X,Y)}$. So it remains to show this is a linearly independent set.  Since
\[
X^aY^bK^c\E^d\F^e=cX^{a+e}Y^{b+d}K^cE^dF^e+\sum\limits_{\substack{d'\leq d, e'\leq e\\ d'+e'<d+e}}c_{d',e'}(X,Y,K)E^{d'}F^{e'},
\]
where $c\neq0$ and $c_{d',e'}(X,Y,K)$ are polynomials in $X,Y,K$, the independence of the set $\{X^aY^bK^c\E^d\F^e|(a,b,c,d,e)\in\Z\times\Z\times\Z\times\Z_+\times\Z_+\}$ follows from the independence of  $\{X^aY^bK^cE^dF^e|(a,b,c,d,e)\in\Z\times\Z\times\Z\times\Z_+\times\Z_+\}$.
\end{proof}

Now we are ready to prove Theorem \ref{center}.

\

{\it Proof of Theorem \ref{center}.}(i)
Consider the localization $\overline{\S}_q^{(X,Y)}$, then we have $Z(\overline{\S}_q)=Z(\overline{\S}_q^{(X,Y)})\cap \overline{\S}_q$ 
Let $Z=\sum r(a,b,c,d,e)X^aY^bK^c\E^d\F^e$ be any nonzero element in $Z(\overline{\S}_q^{(X,Y)})$.

Since
\begin{align*}
0&=KZ-ZK\\
&=\sum rKX^aY^bK^c\E^d\F^e-\sum rX^aY^bK^c\E^d\F^eK\\
&=\sum q^{a-b}rX^aY^bK^{c+1}\E^d\F^e-\sum q^{e-d}rX^aY^bK^{c+1}\E^d\F^e,
\end{align*}
we have for each $(a,b,c,d,e)\in\Z^3\times\Z_+^2$,
\[
r(a,b,c,d,e)(q^{a-b}-q^{e-d})=0.
\]
Therefore, $r(a,b,c,d,e)=0$ unless $e=a-b+d$. So $Z=\sum r(a,b,c,d)X^aY^bK^c\E^d\F^{a-b+d}$.

From
\begin{align*}
0&=XZ-ZX\\
&=\sum rX^{a+1}Y^bK^c\E^d\F^{a-b+d}-\sum rX^aY^bK^c\E^d\F^{a-b+d}X\\
&=\sum rX^{a+1}Y^bK^c\E^d\F^{a-b+d}-\sum q^{c-b}rX^{a+1}Y^bK^c\E^d\F^{a-b+d},
\end{align*}
we know that
\[
r(a,b,c,d)(q^{c-b}-1)=0.
\]
Hence, $r(a,b,c,d)=0$ unless $b=c$, which means that $Z=\sum r(a,b,c)X^aY^bK^b\E^c\F^{a-b+c}$.

From
\begin{align*}
0&=YZ-ZY\\
&=\sum (q^{-a}-q^{a-2b})r(a,b,c)X^aY^{b+1}K^b\E^c\F^{a-b+c},
\end{align*}
 we have $r(a,b,c)=0$ unless $a=b$. So $$Z=\sum r(a,b)(XYK)^a\E^b\F^b=\sum\limits_{a=s}^{t}\sum\limits_{b=0}^nr(a,b)(XYK)^a\E^b\F^b.$$

Following from
\begin{align*}
0=&Z\E-\E Z\\
=&\sum\limits_{a=s}^t\sum\limits_{b=0}^nr(a,b)(XYK)^a\E^b\F^b\E-\sum\limits_{a=s}^t\sum\limits_{b=0}^nr(a,b)\E(XYK)^a\E^b\F^b\\
=&\sum\limits_{a=s}^t\sum\limits_{b=0}^nr(a,b)(XYK)^a\E^b(\E\F^b+(1-q^{-2b})\F^{b-1}(XYK))\\
&-\sum\limits_{a=s}^t\sum\limits_{b=0}^nq^{-2a}r(a,b)(XYK)^a\E^{b+1}\F^b\\
=&\sum\limits_{a=s}^t\sum\limits_{b=0}^nr(a,b)(1-q^{-2a})(XYK)^a\E^{b+1}\F^b\\
&+\sum\limits_{a=s}^t\sum\limits_{b=0}^nr(a,b)q^{2(b-1)}(1-q^{-2b})(XYK)^{a+1}\E^b\F^{b-1},
\end{align*}
we deduce that
\begin{align*}
r(t,n)(1-q^{-2t})&=0,\\
r(t,n)q^{2(n-1)}(1-q^{-2n})&=0.
\end{align*}
Thus, we have $t=n=0$, that is $Z=\sum\limits_{a\leq 0}r(a)(XYK)^a$. So, the first statement of Theorem \ref{center} follows.

(ii) By (i) and Theorem 11.1 in \cite{GK}, for any  complex number $z$, the  center of $\S_q/(C-z)\S_q$ is trivial.  Suppose that $Z=\sum\limits_{(a,b,c,d,e)\in\Z_+^5} r(a,b,c,d,e,C)X^aY^bK^cE^dF^e$ is an element of the center of $\S_q$, where each coefficient $r(a,b,c,d,e,C)$ is a polynomial in $C$.
Suppose that there is some $(a,b,c,d,e)\in\Z_+^5$ with $a+b+c+d+e>0$ such  that the corresponding coefficient  $r(a,b,c,d,e,C)$  is not zero. Choose $z\in \mathbb{C}$ such that  $r(a,b,c,d,e,z)\neq 0$.
Then  the image of $Z$ in $\S_q/(C-z)\S_q$ is not a scalar for any $z\in Z$, which is impossible.
\qed

\section{Some basic results}

In this section, we will give some basic results for our arguments on weight modules.

A classification and explicit description of all simple highest weight $\S_q$-modules was given by Dobrev et al. in \cite{DD,GK}. Using the involution given in \cite{DD},
we can also obtain explicit description of simple lowest weight $\S_q$-modules. Here we recall these results  which are necessary for our arguments.

\begin{theorem}\label{simplehw}
Let $V$ be a simple highest weight $\S_q$-module with central charge $z\in\C$.
\begin{enumerate}[(i)]
\item If $z=0$, then $V$ is a simple highest weight $U_q(\mathfrak{sl}_2)$-module, that is $XV=YV=0$.
\item For $\lambda\in\C^*$ and $z\in\C^*$, let $M(\lambda,z)$ be the Verma module generated by $v_{0,0}$, where $Kv_{0.0}=\lambda v_{0,0}, Cv_{0.0}=zv_{0,0}$. Then  $M(\lambda,z)$ has the basis $\{v_{k,l}:=Y^kF^lv_{0,0}|k,l\in\Z_+\}$ on which the $\S_q$-action is given by
\begin{align*}
K.v_{k,l}&=\lambda q^{-k-2l}v_{k,l},\ \ \  C.v_{k,l}=zv_{k,l},\\
Y.v_{k,l}&=v_{k+1,l}, \ \ \ \ \ F.v_{k,l}=v_{k,l+1},\\
X.v_{k,l}&=-z\frac{q^k-1}{q-1}v_{k-1,l}-\lambda^{-1}q^{k+l-1}[l]_qv_{k+1,l-1},\\
E.v_{k,l}&=\frac{\lambda q^{1-k-l}-\lambda^{-1}q^{k+l-1}}{q-q^{-1}}[l]_qv_{k,l-1} -\frac{q+q^2-q^{2-k}-q^{k+1}}{(1-q^2)(q-1)}zv_{k-2,l}.
\end{align*} The module $M(\lambda,z)$ is simple if $\lambda^2\in\C^*\setminus q^{-3+2\N}$.
For $\lambda$ with $\lambda^2=q^{2d-3}\in q^{-3+2\N}$ and $z\in\C^*$, denote by $N(\lambda,z)$ the unique simple quotient of $M(\lambda,z)$ with basis $\{v_{k,l}\mid k,l\in\Z_+,l\leq d-1\}$, on which the $\S_q$-action is given by
\begin{align*}
K.v_{k,l}&=\lambda q^{-k-2l}v_{k,l},\,\ \ C.v_{k,l}=zv_{k,l},\, \ \ Y.v_{k,l}=v_{k+1,l}, \\
F.v_{k,l}&=\left\{\begin{array}{ll}v_{k,l+1}, & l<d-1,\\
-\sum\limits_{i=0}^{d-1}\Big(\frac{-q^{d-1}}{\lambda z(q+1)}\Big)^{d-i}\binom{d}{i}_qv_{k+2d-2i,i}, & l=d-1,
\end{array}\right.\\
X.v_{k,l}&=-z\frac{q^k-1}{q-1}v_{k-1,l}-\lambda^{-1}q^{k+l-1}[l]_qv_{k+1,l-1},\\
E.v_{k,l}&=\frac{\lambda q^{1-k-l}-\lambda^{-1}q^{k+l-1}}{q-q^{-1}}[l]_qv_{k,l-1} -\frac{q+q^2-q^{2-k}-q^{k+1}}{(1-q^2)(q-1)}zv_{k-2,l}.
\end{align*}
\end{enumerate}
If the central charge of $V$ is nonzero, then $V$ is isomorphic to either some $M(\lambda,z)$ or $N(\lambda,z)$.
\end{theorem}

For any simple $\S_q$-module we have the following property.
\begin{lemma}\label{lonil}
Let $s\in\{E,F,X,Y\}$ and $V$ be a simple $\S_q$-module. If the action of  $s$ on $V$ is not injective, then $s$ acts on $V$ locally nilpotently.
\end{lemma}

To prove this lemma, we need the following equalities.
\begin{lemma}
For $r,s\in\N$, the following equalities hold in the algebra $\S_q$.
\begin{enumerate}[(i)]
\item $E^{(r)}F^{(s)}=\sum\limits_{j\geq0}F^{(s-j)}\left[\begin{smallmatrix}K;2j-r-s\\j\end{smallmatrix}\right]E^{(r-j)},$
 where $E^{(r)}=\frac{E^r}{[r]_q!},F^{(r)}=\frac{F^r}{[r]_q!}$,\\
  $\left[\begin{smallmatrix}K;r\\s\end{smallmatrix}\right]=\prod\limits_{j=1}^s\frac{Kq^{r-j+1}-K^{-1}q^{-r+j-1}}{q^j-q^{-j}}$.
\item $E^{r+1}Y^r=(q^{-r}EY+q[r]_qX)E^rY^{r-1}$.
\item $X^{r}Y^s=q^{rs}Y^sX^r+\sum\limits_{i=1}^{\min(r,s)}(-1)^iq^{rs+i}\prod\limits_{j=0}^{i-1}\frac{(1-q^{-r+j})(1-q^{-s+j})}{(q-1)(q^{j+1}-1)}C^iY^{s-i}X^{r-i}.$
\item $XF^r=F^rX-q^{r-1}[r]_qYF^{r-1}K^{-1}$.
\end{enumerate}
\end{lemma}
\begin{proof}
(i) comes from the formula (a2) on page 103 in \cite{L}.

(ii) follows from induction on $r$.
\begin{align*}
E^{r+2}Y^{r+1}&=E(E^{r+1}Y^{r})Y\\
&=E(q^{-r}EY+q[r]_qX)E^rY^r\\
&=(q^{-r}E^2Y+q^2[r]_qXE)E^rY^r\\
&=(q^{-r}E(q^{-1}YE+X)+q^2[r]_qXE)E^rY^r\\
&=(q^{-r-1}EY+(q^{-r+1}+q^2[r]_q)X)E^{r+1}Y^r\\
&=(q^{-r-1}EY+q[r+1]_qX)E^{r+1}Y^r.
\end{align*}

(iii) Following from induction on $s$ and $r$, it is easy to get
\begin{align*}
XY^s&=q^sY^sX-\frac{q^s-1}{q-1}CY^{s-1},\\
X^rY&=q^rYX^r-\frac{q^r-1}{q-1}CX^{r-1}.
\end{align*}
Replacing $q$ by $q^{-1}$ and $C$ by $-q^{-1}C$, we may assume that $s<r$. Then the induction follows from
\begin{align*}
X^{r+1}Y^s=&X\Big(q^{rs}Y^sX^r+\sum\limits_{i=1}^s(-1)^iq^{rs+i}\prod\limits_{j=0}^{i-1}\frac{(1-q^{-r+j})(1-q^{-s+j})}{(q-1)(q^{j+1}-1)}C^iY^{s-i}X^{r-i}\Big)\\
=&q^{rs}XY^sX^r+\sum\limits_{i=1}^s(-1)^iq^{rs+i}\prod\limits_{j=0}^{i-1}\frac{(1-q^{-r+j})(1-q^{-s+j})}{(q-1)(q^{j+1}-1)}C^iXY^{s-i}X^{r-i}\\
=&q^{rs}\Big(q^sY^sX^{r+1}-\frac{q^s-1}{q-1}CY^{s-1}X^r\Big)\\
&+\sum\limits_{i=1}^{s-1}(-1)^iq^{rs+i}\prod\limits_{j=0}^{i-1}\frac{(1-q^{-r+j})(1-q^{-s+j})}{(q-1)(q^{j+1}-1)}C^i\Big(q^{s-i}Y^{s-i}X-\frac{q^{s-i}-1}{q-1}CY^{s-i-1}\Big)X^{r-i}\\
&+(-1)^sq^{(r+1)s}\prod\limits_{j=0}^{s-1}\frac{(1-q^{-r+j})(1-q^{-s+j})}{(q-1)(q^{j+1}-1)}C^sX^{r+1-s}\\
=&q^{(r+1)s}Y^sX^{r+1}-q^{rs}\frac{q^s-1}{q-1}CY^{s-1}X^r\\
&+\sum\limits_{i=1}^s(-1)^iq^{(r+1)s}\prod\limits_{j=0}^{i-1}\frac{(1-q^{-r+j})(1-q^{-s+j})}{(q-1)(q^{j+1}-1)}C^iY^{s-i}X^{r+1-i}\\
&+\sum\limits_{i=1}^{s-1}(-1)^{i+1}q^{rs+i}\frac{q^{s-i}-1}{q-1}\prod\limits_{j=0}^{i-1}\frac{(1-q^{-r+j})(1-q^{-s+j})}{(q-1)(q^{j+1}-1)}C^{i+1}Y^{s-i-1}X^{r-i}\\
=&q^{(r+1)s}Y^sX^{r+1}-q^{(r+1)s+1}\frac{(1-q^{-r-1})(1-q^{-s})}{(q-1)^2}CY^{s-1}X^r\\
&+\sum\limits_{i=2}^s(-1)^i\Big(q^{(r+1)s}\frac{(1-q^{-r+i-1})(1-q^{-s+i-1})}{(q-1)(q^i-1)}+q^{rs+i-1}\frac{q^{s-i+1}-1}{q-1}\Big)\\
&\cdot\prod\limits_{j=0}^{i-2}\frac{(1-q^{-r+j})(1-q^{-s+j})}{(q-1)(q^{j+1}-1)}C^iY^{s-i}X^{r+1-i}\\
=&q^{(r+1)s}Y^sX^{r+1}+\sum\limits_{i=1}^{s}(-1)^iq^{(r+1)s+i}\prod\limits_{j=0}^{i-1}\frac{(1-q^{-r-1+j})(1-q^{-s+j})}{(q-1)(q^{j+1}-1)}C^iY^{s-i}X^{r+1-i};
\end{align*}
and
\begin{align*}
&X^rY^{s+1}\\
=&q^{rs}Y^sX^rY+\sum\limits_{i=1}^s(-1)^iq^{rs+i}\prod\limits_{j=0}^{i-1}\frac{(1-q^{-r+j})(1-q^{-s+j})}{(q-1)(q^{j+1}-1)}C^iY^{s-i}X^{r-i}Y\\
=&q^{rs}Y^s\Big(q^rYX^r-\frac{q^r-1}{q-1}CX^{r-1}\Big)\\
&+\sum\limits_{i=1}^s(-1)^iq^{rs+i}\prod\limits_{j=0}^{i-1}\frac{(1-q^{-r+j})(1-q^{-s+j})}{(q-1)(q^{j+1}-1)}C^iY^{s-i}\Big(q^{r-i}YX^{r-i}-\frac{q^{r-i}-1}{q-1}CX^{r-i-1}\Big)\\
=&q^{r(s+1)}Y^{s+1}X^r-q^{rs}\frac{q^r-1}{q-1}CY^sX^{r-1}\\
&+\sum\limits_{i=1}^s(-1)^iq^{r(s+1)}\prod\limits_{j=0}^{i-1}\frac{(1-q^{-r+j})(1-q^{-s+j})}{(q-1)(q^{j+1}-1)}C^iY^{s+1-i}X^{r-i}\\
&+\sum\limits_{i=1}^s(-1)^{i+1}q^{rs+i}\frac{q^{r-i}-1}{q-1}\prod\limits_{j=0}^{i-1}\frac{(1-q^{-r+j})(1-q^{-s+j})}{(q-1)(q^{j+1}-1)}C^{i+1}Y^{s-i}X^{r-i-1}\\
=&q^{r(s+1)}Y^{s+1}X^r-q^{r(s+1)+1}\frac{(1-q^{-r})(1-q^{-s-1})}{(q-1)^2}CY^sX^{r-1}\\
&+\sum\limits_{i=2}^s(-1)^iq^{r(s+1)+i}\Big(q^{-i}\frac{(1-q^{-r+i-1})(1-q^{-s+i-1})}{(q-1)(q^i-1)}+q^{-r-1}\frac{q^{r-i+1}-1}{q-1}\Big)\\
&\cdot\prod\limits_{j=0}^{i-2}\frac{(1-q^{-r+j})(1-q^{-s+j})}{(q-1)(q^{j+1}-1)}C^iY^{s+1-i}C^{r-i}\\
&+(-1)^{s+1}q^{(r+1)s}\frac{q^{r-s}-1}{q-1}\prod\limits_{j=0}^{s-1}\frac{(1-q^{-r+j})(1-q^{-s+j})}{(q-1)(q^{j+1}-1)}C^{s+1}X^{r-s-1}\\
=&q^{r(s+1)}Y^{s+1}X^r+\sum\limits_{i=1}^{s+1}(-1)^iq^{r(s+1)+i}\prod\limits_{j=0}^{i-1}\frac{(1-q^{-r+j})(1-q^{-s+j})}{(q-1)(q^{j+1}-1)}C^iY^{s+1-i}X^{r-i}.
\end{align*}
(iv) follows from induction on $r$:
\begin{align*}
XF^{r+1}&=XF^rF\\
&=(F^rX-q^{r-1}[r]_qYF^{r-1}K^{-1})F\\
&=F^rXF-q^{r+1}[r]_qYF^{r}K^{-1}\\
&=F^r(FX-YK^{-1})-q^{r+1}[r]_qYF^{r}K^{-1}\\
&=F^{r+1}X-q^r[r+1]_qYF^{r}K^{-1}.\tag*{\qedhere}
\end{align*}
\end{proof}

{\it Proof of Lemma \ref{lonil}.} We only need to prove the lemma for $u=E,X$. By assumption, there exists a nonzero vector $v\in V$ such that $u.v=0$.

(i) If $u=E$, then we need to show that $E$ acts nilpotently on $X^aK^bF^cY^dv$. This follows from the following computation. For $m\gg0$, we have
\begin{align*}
&\frac{1}{[m]_q!}\frac{1}{[c]_q!}E^mX^aK^bF^cY^dv\\
=&q^{ma-2mb}X^aK^bE^{(m)}F^{(c)}Y^dv\\
=&q^{ma-2mb}X^aK^b\sum\limits_{j\geq0}F^{(c-j)}\left[\begin{smallmatrix}K;2j-m-c\\j\end{smallmatrix}\right]E^{(m-j)}Y^dv\\
=&\sum\limits_{j\geq0}q^{ma-2mb}\prod\limits_{i=0}^d[m-j-i]_qX^aK^bF^{(c-j)}\left[\begin{smallmatrix}K;2j-m-c\\j\end{smallmatrix}\right]E^{(m-j-d-1)}E^{d+1}Y^dv\\
=&\sum\limits_{j\geq0}q^{ma-2mb}\prod\limits_{i=0}^d[m-j-i]_qX^aK^bF^{(c-j)}\left[\begin{smallmatrix}K;2j-m-c\\j\end{smallmatrix}\right]E^{(m-j-d-1)}(q^{-d}EY+q[d]_qX)E^dY^{d-1}v\\
=&\sum\limits_{j\geq0}q^{ma-2mb}\prod\limits_{i=0}^d[m-j-i]_qX^aK^bF^{(c-j)}\left[\begin{smallmatrix}K;2j-m-c\\j\end{smallmatrix}\right]E^{(m-j-d-1)}\prod\limits_{i=1}^{d}(q^{-i}EY+q[i]_qX)Ev\\
=&0.
\end{align*}

(ii) For $u=X$, similarly it suffices to show that $X$ acts nilpotently on $E^aK^bY^cF^dv$. Since for $m\gg0$, we have
\begin{align*}
X^mE^aK^bY^cF^dv&=q^{-ma-mb}E^aK^bX^mY^cF^dv\\
&=q^{mc-ma-mb}E^aK^bY^cX^mF^dv+\sum\limits_{i=1}^cc(i,s,r)q^{-ma-mb}E^aK^bY^{c-i}X^{m-i}F^dv,
\end{align*}
where $c(i,s,r)=(-1)^iq^{rs+i}\prod\limits_{j=0}^{i-1}\frac{(1-q^{-r+j})(1-q^{-s+j})}{(q-1)(q^{j+1}-1)}$. We only need to show that $X$ acts nilpotently on $F^dv$ for any $d\in\N$, which follows by the following fact: if $X^rF^sv=0$, then $X^{r+2}F^{s+1}v=0$.

Indeed, we have
\begin{align*}
X^{r+2}F^{s+1}v&=X^{r+1}(F^{s+1}X-q^s[s+1]_qYF^{s}K^{-1})v\\
&=-q^s[s+1]_qX^{r+1}YF^{s}K^{-1}v\\
&=-q^s[s+1]_q\Big(q^{r+1}YX^{r+1}F^s-\frac{q^{r+1}-1}{q-1}X^{r}\Big)F^sK^{-1}v\\
&=0.\tag*{\qed}
\end{align*}

For $L\in\cN$, we have the following crucial property.
\begin{theorem}\label{wdim}
Let $L\in\cN$. Then $\mathrm{supp}(L)=\lambda q^\Z$ for some $\lambda\in\C^*$, and $\dim L_{\lambda q^i}=\dim L_{\lambda q^j}$, for all $i,j\in\Z$.
\end{theorem}
\begin{proof}
Suppose $L=\oplus_{i\in\Z}L_{\lambda q^i}$. If there exists some $i\in\Z$ such that $\dim L_{\lambda q^i}>\dim L_{\lambda q^{i+1}}$,
 then the action of $X$ on $L$ is not injective, and hence $X$ acts on $L$ locally nilpotently.

If $\dim L_{\lambda q^{i-1}}>\dim L_{\lambda q^{i+1}}$, then the action of $E$ on $L$ is not injective and $E$ acts locally nilpotently on $L$. Since $EX=qXE$, there exists $v\in L$ such that $Xv=Ev=0$, which means that $L$ is a highest weight module. This contradicts with our assumption.

Therefore, we have $\dim L_{\lambda q^{i-1}}\leq\dim L_{\lambda q^{i+1}}<\dim L_{\lambda q^i}$. From this we know that $Y$ acts locally nilpotently on $L$. If $L$ has zero central charge, then following from $qYX-XY=C$, there exists nonzero $v\in L$ such that $Xv=Yv=0$. Hence $XL=YL=0$, which is impossible. If $L$ has nonzero central charge $z$, then there exists nonzero $v\in L$ with $Yv=0, X^mv\neq0$ and $X^{m+1}v=0$. So, we have
\[0=q^{m+1}YX^{m+1}v=(X^{m+1}Y+\frac{q^{m+1}-1}{q-1}X^mC)v=z\frac{q^{m+1}-1}{q-1}X^mv\neq0,\]
which is a contradiction.
\end{proof}

\section{Twisting functor}

In this section, we recall the technique of localization which
 was used  by Mathieu to classify  simple weight modules over simple Lie algebras, see \cite{M}.

Since for $u\in\{F,Y\}$, $\{u^i|i\in\Z^+\}$ is an Ore set for $\S_q$, we have the following automorphism.
\begin{proposition}
For $b\in\C^*$, the assignment
\begin{align}
&\Theta_b^{(F)}(F^{\pm1})=F^{\pm1},\ \Theta_b^{(F)}(Y)=Y,\  \Theta_b^{(F)}(C)=C,\\
&\Theta_b^{(F)}(X)=X-\frac{b-1}{q^2-1}F^{-1}YK^{-1},\ \Theta_b^{(F)}(K^i)=b^{-i}K^i,\\
&\Theta_b^{(F)}(E)=E+\frac{1-b^{-1}}{(q^2-1)(q-q^{-1})}KF^{-1}-\frac{1-b}{(q^{-2}-1)(q-q^{-1})}K^{-1}F^{-1}
\end{align}
extends uniquely to an automorphism $\Theta_b^{(F)}: \S_q^{(F)}\rightarrow\S_q^{(F)}$ and
the assignment
\begin{align}
&\Theta_b^{(Y)}(Y^{\pm1})=Y^{\pm1},\ \Theta_b^{(Y)}(F)=F,\ \Theta_b^{(Y)}(C)=C,\\
&\Theta_b^{(Y)}(X)=bX-\frac{b-1}{q-1}CY^{-1},\ \Theta_b^{(Y)}(K^i)=b^{-i}K^i,\\
&\Theta_b^{(Y)}(E)=b^{-1}E+\frac{b-b^{-1}}{q-q^{-1}}Y^{-1}X-\frac{b+qb^{-1}-q-1}{(q-q^{-1})(q-1)}CY^{-2}
\end{align}
extends uniquely to an automorphism $\Theta_b^{(Y)}: \S_q^{(Y)}\rightarrow\S_q^{(Y)}$.
\end{proposition}

\begin{proof}

First we  prove the proposition for $b=q^{2i}, i\in\N$.  We claim that formulas (4.1-4.3)  correspond to restriction of the conjugation automorphism
 $a\mapsto Y^{-2i}aY^{2i}$ of $\S_q^{(Y)}$. To prove this we proceed by induction on $i$. The base $i=0$ is immediate. Let us check the induction step. For $a=Y,F,C$, the formulas are obvious, we only need to check the formulas for $a=K^n, X$ and $E$.
For $u=F$, we have
\begin{align*}
F^{-1}\Theta_{q^{2i}}^{(F)}(F^{\pm1})F&=F^{-1}F^{\pm1}F=F^{\pm1},\\
\\
F^{-1}\Theta_{q^{2i}}^{(F)}(Y)F&=F^{-1}YF=Y,\\
\\
F^{-1}\Theta_{q^{2i}}^{(F)}(K^n)F&=F^{-1}q^{-2ni}K^nF=q^{-2ni}F^{-1}q^{-2n}FK^n=q^{-2n(i+1)}K^n,\\
\\
F^{-1}\Theta_{q^{2i}}^{(F)}(X)F&=F^{-1}\Big(X-\frac{q^{2i}-1}{q^2-1}F^{-1}YK^{-1}\Big)F\\
&=F^{-1}(FX-YK^{-1})-\frac{q^{2(i+1)}-q^2}{q^2-1}F^{-1}YK^{-1}\\
&=X-\frac{q^{2(i+1)}-1}{q^2-1}F^{-1}YK^{-1},
\end{align*}
\begin{align*}
&F^{-1}\Theta_{q^{2i}}^{(F)}(E)F\\
=&F^{-1}\Big(E+\frac{1-q^{-2i}}{(q^2-1)(q-q^{-1})}KF^{-1}-\frac{1-q^{2i}}{(q^{-2}-1)(q-q^{-1})}K^{-1}F^{-1}\Big)F\\
=&F^{-1}EF+\frac{1-q^{-2i}}{(q^2-1)(q-q^{-1})}F^{-1}K-\frac{1-q^{2i}}{(q^{-2}-1)(q-q^{-1})}F^{-1}K^{-1}\\
=&E+F^{-1}\frac{K-K^{-1}}{q-q^{-1}}+\frac{1-q^{-2i}}{(q^2-1)(q-q^{-1})}F^{-1}K-\frac{1-q^{2i}}{(q^{-2}-1)(q-q^{-1})}F^{-1}K^{-1}\\
=&E+\frac{1-q^{-2(i+1)}}{(q^2-1)(q-q^{-1})}KF^{-1}-\frac{1-q^{2(i+1)}}{(q^{-2}-1)(q-q^{-1})}K^{-1}F^{-1};
\end{align*}
For $u=Y$,
We have that
\begin{align*}
Y^{-2}\Theta_{q^{2i}}^{(Y)}(K^n)Y^2&=Y^{-2}q^{-2ni}K^nY^2=Y^{-2}q^{-2ni}q^{-2n}Y^2K^n=q^{-2n(i+1)}K^n,\\
\\
Y^{-2}\Theta_{q^{2i}}^{(Y)}(X)Y^2&=Y^{-2}\Big(q^{2i}X-\frac{q^{2i}-1}{q-1}CY^{-1}\Big)Y^2\\
&=q^{2i}Y^{-2}XY^2-\frac{q^{2i}-1}{q-1}CY^{-1}\\
&=q^{2i}\Big(q^2X-\frac{q^{2}-1}{q-1}CY^{-1}\Big)-\frac{q^{2i}-1}{q-1}CY^{-1}\\
&=q^{2(i+1)}X-\frac{q^{2(i+1)}-1}{q-1}CY^{-1},
\end{align*}
\begin{align*}
&Y^{-2}\Theta_{q^{2i}}^{(Y)}(E)Y^2\\
=&Y^{-2}\Big(q^{-2i}E+[2i]_qY^{-1}X-\frac{q^{2i}+q^{-2i+1}-q-1}{(q-q^{-1})(q-1)}CY^{-2}\Big)Y^2\\
=&q^{-2i}Y^{-2}EY^2+[2i]_qY^{-3}XY^2-\frac{q^{2i}+q^{-2i+1}-q-1}{(q-q^{-1})(q-1)}CY^{-2}\\
=&q^{-2i}(q^{-2}E+[2]_qY^{-1}X-CY^{-2})+[2i]_qY^{-1}\Big(q^2X-\frac{q^2-1}{q-1}CY^{-1}\Big)\\
&-\frac{q^{2i}+q^{-2i+1}-q-1}{(q-q^{-1})(q-1)}CY^{-2}\\
=&q^{-2(i+1)}E+[2(i+1)]_qY^{-1}X-\frac{q^{2(i+1)}+q^{-2i-1}-q-1}{(q-q^{-1})(q-1)}CY^{-2}.
\end{align*}

Hence, when  $b=q^{2i}, i\in\N$,  we have that $\Theta_b^{(Y)}(xy)=\Theta_b^{(Y)}(x)\Theta_b^{(Y)}(y)$ for any $x,y\in \S_q^{(Y)}$.
We can see that  if $p(x)$ is a Laurent polynomial such that $p(q^{2i})=0$ for all $i\in\N$, then the polynomial $p(x)$ is zero. Thus
$\Theta_b^{(Y)}(xy)=\Theta_b^{(Y)}(x)\Theta_b^{(Y)}(y)$ for any $b\in \mathbb{C}\setminus\{0\}$, $x,y\in \S_q^{(Y)}$. So $\Theta_b^{(Y)}$ is an automorphism of
$\S_q^{(Y)}$ for any $b\in \mathbb{C}^*$.
\end{proof}

\begin{proposition}
For $u\in\{F,Y\}$ and $x,y\in\C^*$, we have $\Theta_x^{(u)}\Theta_y^{(u)}=\Theta_{xy}^{(u)}$.
\end{proposition}
\begin{proof}
 We only need to check $\Theta_x^{(u)}\Theta_y^{(u)}(X)=\Theta_{xy}^{(u)}(X)$ and $\Theta_x^{(u)}\Theta_y^{(u)}(E)=\Theta_{xy}^{(u)}(E)$ since the others are trivial. This is done by the following computations.
\begin{align*}
\Theta_x^{(F)}\Theta_y^{(F)}(X)&=\Theta_x^{(F)}(X-\frac{y-1}{q^2-1}F^{-1}YK^{-1})\\
&=X-\frac{x-1}{q^2-1}F^{-1}YK^{-1}-x\frac{y-1}{q^2-1}F^{-1}YK^{-1}\\
&=X-\frac{xy-1}{q^2-1}F^{-1}YK^{-1}\\
&=\Theta_{xy}^{(F)}(X),\\
\\
\Theta_x^{(F)}\Theta_y^{(F)}(E)=&\Theta_x^{(F)}(E+\frac{1-y^{-1}}{(q^2-1)(q-q^{-1})}KF^{-1}-\frac{1-y}{(q^{-2}-1)(q-q^{-1})}K^{-1}F^{-1})\\
=&E+\frac{1-x^{-1}}{(q^2-1)(q-q^{-1})}KF^{-1}-\frac{1-x}{(q^{-2}-1)(q-q^{-1})}K^{-1}F^{-1}\\
&+\frac{x^{-1}(1-y^{-1})}{(q^2-1)(q-q^{-1})}KF^{-1}-\frac{x(1-y)}{(q^{-2}-1)(q-q^{-1})}K^{-1}F^{-1}\\
=&E+\frac{1-x^{-1}y^{-1}}{(q^2-1)(q-q^{-1})}KF^{-1}-\frac{1-xy}{(q^{-2}-1)(q-q^{-1})}K^{-1}F^{-1}\\
=&\Theta_{xy}^{(F)}(E),\\
\\
\Theta_x^{(Y)}\Theta_y^{(Y)}(X)&=\Theta_x^{(Y)}\Big(yX-\frac{y-1}{q-1}CY^{-1}\Big)\\
&=y\Theta_x^{(Y)}(X)-\frac{y-1}{q-1}CY^{-1}\\
&=xyX-y\frac{x-1}{q-1}CY^{-1}-\frac{y-1}{q-1}CY^{-1}\\
&=xyX-\frac{xy-1}{q-1}CY^{-1}\\
&=\Theta_{xy}^{(Y)}(X),\\
\\
\Theta_x^{(Y)}\Theta_y^{(Y)}(E)=&\Theta_x^{(Y)}\Big(y^{-1}E+\frac{y-y^{-1}}{q-q^{-1}}Y^{-1}X-\frac{y+qy^{-1}-q-1}{(q-q^{-1})(q-1)}CY^{-2}\Big)\\
=&y^{-1}\Theta_x^{(Y)}(E)+\frac{y-y^{-1}}{q-q^{-1}}Y^{-1}\Theta_x^{(Y)}(X)-\frac{y+qy^{-1}-q-1}{(q-q^{-1})(q-1)}CY^{-2}\\
=&y^{-1}\Big(x^{-1}E+\frac{x-x^{-1}}{q-q^{-1}}Y^{-1}X-\frac{x+qx^{-1}-q-1}{(q-q^{-1})(q-1)}CY^{-2}\Big)\\
&+\frac{y-y^{-1}}{q-q^{-1}}Y^{-1}\Big(xX-\frac{x-1}{q-1}CY^{-1}\Big)-\frac{y+qy^{-1}-q-1}{(q-q^{-1})(q-1)}CY^{-2}\\
=&(xy)^{-1}E+\frac{xy-x^{-1}y^{-1}}{q-q^{-1}}Y^{-1}X-\frac{xy+qx^{-1}y^{-1}-q-1}{(q-q^{-1})(q-1)}CY^{-2}.\tag*{\qedhere}
\end{align*}
\end{proof}

Now we can define Mathieu's twisting functor in our situation. For $b\in\C^*, u\in\{F,Y\}$, the {\it twisting functor} $B_b^{(u)}:\S_q$-$\mathrm{Mod}\rightarrow\S_q$-$\mathrm{Mod}$ is defined as composition of the following functors:
\begin{enumerate}[(i)]
\item the induction functor $\mathrm{Ind}_{\S_q}^{\S_q^{(u)}}:=\S_q^{(u)}\otimes_{\S_q}-$;
\item twisting the $\S_q^{(u)}$-action by $\Theta_b^{(u)}$;
\item the restriction functor $\mathrm{Res}_{\S_q}^{\S_q^{(u)}}$.
\end{enumerate}

The following two results are similar as Lemma 10 and Proposition 11 in \cite{D}
\begin{lemma}
Let $u\in\{F,Y\}$. Let $V$ ba an $\S_q$-module on which $u$ acts bijectively and $W$ be an $\S_q^{(Y)}$-module. Then
\begin{enumerate}[(i)]
\item $\mathrm{Ind}_{\S_q}^{\S_q^{(u)}}\circ\mathrm{Res}_{\S_q}^{\S_q^{(u)}}(W)\cong W$.
\item $\mathrm{Res}_{\S_q}^{\S_q^{(u)}}\circ\mathrm{Ind}_{\S_q}^{\S_q^{(u)}}(V)\cong V$.
\end{enumerate}
\end{lemma}

\begin{proposition}
For $x,y\in\C^*, u\in\{F,Y\}$, we have $B_x^{(u)}\circ B_y^{(u)}\cong B_{xy}^{(u)}$.
\end{proposition}

\begin{proposition}\label{twist}
Let $L$ be a simple $\S_q$-module and $u\in\{F,Y\}$. Then $B_b^{(u)}(L)$ is a simple $\S_q$-module for any $b\in\C^*$.
\end{proposition}
\begin{proof}
This follows from the fact that if $W\subseteq B_b^{(u)}(L)$ is a submodule, then $B_{b^{-1}}^{(u)}(W)\subseteq B_{b^{-1}}^{(u)}(B_b^{(u)}(L))\cong L$ is a submodule.
\end{proof}

\section{Classification of simple modules}

In this section, we will clssify all simple weight $\S_q$-modules with finite dimensional weight spaces. First, we have
\begin{lemma}\label{inj}
Let $V$ be a simple weight $\S_q$-module with finite dimensional weight spaces.
\begin{enumerate}[(i)]
\item If $E$ acts locally nilpotently on $V$, then $V$ is a highest weight module. If $F$ acts locally nilpotently on $V$, then $V$ is a lowest weight module.
\item Suppose in addition that $V$ has nonzero central charge. If $X$ acts locally nilpotently on $V$, then $V$ is a highest weight module. If $Y$ acts locally nilpotently on $V$, then $V$ is a lowest weight module.
\end{enumerate}
\end{lemma}

\begin{proof}
(i) We only need to prove the claim for the element $E$, the other case is similar. Suppose $E$ acts on $V$ locally nilpotently and $V$ is not a highest weight module, then $V$ is not a lowest weight module either, for otherwise $E$ and $F$ would both act locally nilpotently on $V$ and hence $V$ is a direct sum of finite dimensional modules when restricted to $U_q(\mathfrak{sl}_2)$. Since $V$ has finite dimensional weight spaces, $V$ is finite dimensional, and hence a highest weight module.

Therefore, $V$ is either a dense $U_q(\mathfrak{sl}_2)$-module or is in $\cN$. However, $E$ does not act locally nilpotently on simple dense $U_q(\mathfrak{sl}_2)$-modules, so $V$ is in $\cN$. By Theorem \ref{wdim}, we have $\mathrm{supp}(V)=\lambda q^{\Z}$ for some $\lambda\in\C^*$ and all nonzero weight spaces of $V$ have the same dimension. So $V$ has finite length as a $U_q(\mathfrak{sl}_2)$-module. The only simple weight $U_q(\mathfrak{sl}_2)$-modules on which $E$ acts locally nilpotently are highest weight modules, therefore, as a $U_q(\mathfrak{sl}_2)$-module, $V$ has a finite filtration with subquotients being highest weight modules. Therefore, $V$ must have a highest weight, a contradiction. This proves claim (i).

(ii) Again we only need to consider the claim for the element $X$. Take any nonzero weight vector $v$ with weight $\lambda$ such that $Xv=0$. By claim (i) we may assume that $E$ acts injectively on $V$. Since $EX=qXE$, we must have for any $i\in\Z_+$, $v_i=E^iv\neq0$ and $Xv_i=0$.

We claim that the action of $Y$ on $V$ is injective for otherwise it would be locally nilpotent and then $Y^kv=0$ for some $k\in\Z_+$ while $Y^{k-1}v\neq0$. Thus, we have
\[
0=XY^kv=-z\frac{q^k-1}{q-1}Y^{k-1}v\neq0,
\]
which is a contradiction.

Finally, we show that $\{Y^{2i}v_i|i\in\Z_+\}$ is an infinite set of linearly independent elements. First, \[X^{2k}Y^{2i}v_i=\sum\limits_{t=1}^{\min(2k,2i)}(-1)^tq^{4ki+t}\prod\limits_{j=0}^{t-1}\frac{(1-q^{-2k+j})(1-q^{-2i+j})}{(q-1)(q^{j+1}-1)}C^tY^{2i-t}X^{2k-t}v_i\]
tells us that $X^{2k}Y^{2i}v_i=0$ only if $k>i$. Hence, $Y^{2i}v_i\neq Y^{2j}v_j$ whenever $i\neq j$. Now suppose $\sum\limits_{i=0}^ka_iY^{2i}v_i=0$, then we have
\[
0=X^{2k}(\sum\limits_{i=0}^ka_iY^{2i}v_i)=a_k(-z)^{2k}q^{4k^2+2k}\prod\limits_{j=0}^{2k-1}\frac{(1-q^{-2k+j})(1-q^{-2k+j})}{(q-1)(q^{j+1}-1)}v_k,
\]
which is impossible. Hence, we have infinitely linearly many independent weight vectors of weight $\lambda$, a contradiction.
\end{proof}

\begin{corollary}\label{bij}
Let $L\in\cN$. Then both $E$ and $F$ acts bijectively on $L$. If in addition $L$ has nonzero central charge, then also both $X$ and $Y$ acts bijectively on $L$.
\end{corollary}
\begin{proof}
From Lemma \ref{inj} it follows that the actions of $E,F,X$ and $Y$ on $L$ are not locally nilpotent, hence are   injective. By Theorem \ref{wdim}, these actions restrict to injective actions
between finite-dimensional vector spaces of the same dimension. Therefore they all are bijective.
\end{proof}

\begin{lemma}\label{hwv}
Let $L\in\cN$.
\begin{enumerate}[(i)]
\item There exists $b\in\C^*$ and $0\neq v\in B_b^{(F)}(L)$ such that $Ev=0$.
\item Assume that $C L\neq 0$. Then there exists $b\in\C^*$ and $0\neq v\in B_b^{(Y)}(L)$ such that $Xv=0$.
\end{enumerate}
\end{lemma}

\begin{proof}
(i) By Corollary \ref{bij}, both $E$ and $F$ act bijectively on $L$, and hence $EF: L_\lambda\rightarrow L_\lambda$ is bijective for any $\lambda\in\mathrm{supp}(L)$. So there exists $v\in L_\lambda$ such that $EFv=av$ for some $a\in\C^*$.

Since
\begin{align*}
&\Theta_b^{(F)}(E)Fv\\
=&(E+\frac{1-b^{-1}}{(q^2-1)(q-q^{-1})}KF^{-1}-\frac{1-b}{(q^{-2}-1)(q-q^{-1})}K^{-1}F^{-1})Fv\\
=&\Big(a+\lambda\frac{1-b^{-1}}{(q^2-1)(q-q^{-1})}-\lambda^{-1}\frac{1-b}{(q^{-2}-1)(q-q^{-1})}\Big)v,
\end{align*}
we can choose $b\in\C^*$ such that $\Theta_b^{(F)}(E)Fv=0$.

(ii) Let $L_\lambda$ be a weight space of $L$. Suppose $\dim L_\lambda=m$. Since both $X$ and $Y$ act bijectively on $L$, $XY$ is a bijective operator on $L_{\lambda}$ and hence there exists $v\in L_\lambda$ such that $XYv=a_1v$ for some $a_1\in\C^*$. Suppose $C$ acts like $c\in\C^*$ on $L$. If we cannot find $a_1\neq\frac{c}{q-1}$, then we can find a basis $\{v_1,\cdots,v_m\}$ for $L_\lambda$ such that
\begin{align*}
XYv_1&=\frac{c}{q-1}v_1,\\
XYv_i&=\frac{c}{q-1}v_i+\epsilon_iv_{i-1}, \, 2\leq i\leq m,
\end{align*}
where $\epsilon_i\in\{0,1\}, 2\leq i\leq m$.

Because $X$ acts on $L$ bijectively, so $\{X^kv_i|i=1,\cdots,m\}$ is a basis for $L_{\lambda q^k}$ for any $k\in\Z_+$. Similarly, $\{Ev_i|1\leq i\leq m\}$ is a basis for $L_{\lambda q^2}$. Suppose that $Ev_i=\sum\limits_{j=1}^ma_{ij}X^2v_j$, then
\begin{align*}
XYEv_i&=XY(\sum\limits_{j=1}^ma_{ij}X^2v_j)\\
&=q^{-2}\sum\limits_{j=1}^ma_{ij}X^2XYv_j+\frac{1-q^{-2}}{q-1}c\sum\limits_{j=1}^ma_{ij}X^2v_j\\
&=\frac{q^{-2}c}{q-1}Ev_i+\frac{1-q^{-2}}{q-1}c\sum\limits_{j=1}^ma_{ij}X^2v_j+q^{-2}\sum\limits_{j=2}^ma_{ij}\epsilon_jX^2v_{j-1}\\
&=\frac{c}{q-1}Ev_i+q^{-2}\sum\limits_{j=2}^ma_{ij}\epsilon_jX^2v_{j-1}.
\end{align*}
On the other hand, we have
\[
XYEv_i=qX(EY-X)v_i=\left\{\begin{array}{ll}
\frac{c}{q-1}Ev_1-qX^2v_1, &\text{ if } i=1;\\
\frac{c}{q-1}Ev_i+\epsilon_iEv_{i-1}-qX^2v_i, &\text{ if } i\geq 2.
\end{array}\right.\]
Hence, we have
\begin{align*}
\sum\limits_{j=2}^ma_{1j}\epsilon_jX^2v_{j-1}&=-q^3X^2v_1,\\
\sum\limits_{j=2}^ma_{ij}\epsilon_jX^2v_{j-1}&=q^2\epsilon_i \sum\limits_{j=1}^ma_{i-1,j}X^2v_j-q^3X^2v_i,\,\ 2\leq i\leq m.
\end{align*}
So
\begin{align*}
\epsilon_2a_{12}&=-q^3,\\
\epsilon_{i+1}a_{i,i+1}&=q^2\epsilon_ia_{i-1,i}-q^3,\,\ 2\leq i\leq m-1,\\
0&=q^2\epsilon_ma_{m-1,m}-q^3.
\end{align*}
Therefore, we get
\[
0=q^2\epsilon_ma_{m-1,m}-q^3=-q^3\frac{1-q^{2m}}{1-q^2},
\]
which implies that $q^{2m}=1$, this contradicts with our assumption that $q$ is not a root of unity.

 So, we can find $a_1$ which is not $\frac{c}{q-1}$. Take $b=\big(1-(q-1)c^{-1}a_1\big)^{-1}$, then
\[
\Theta_b^{(Y)}(X)Yv=0.\qedhere
\]
\end{proof}

\begin{proposition}\label{hwsub}
Let $V$ be a uniformly bounded weight $\S_q$-module with $\mathrm{supp}(V)\subseteq\lambda q^{\Z}$ for some $\lambda\in\C^*$.
If there is some $0\neq v\in V$ such that $Ev=0$ or $Xv=0$, then $V$ has a simple highest weight submodule.
\end{proposition}
\begin{proof}
By assumption, for $u=E,X$,
\[
W_u:=\{v\in V|u^nv=0\mbox{ for some } n\in\N\}\neq0.
\]
By Lemma \ref{lonil}, $W_u$ is a submodule. Since $V$ has finite length, any simple submodule of $W_u$ is a highest weight module.
\end{proof}


Now we are ready to prove our main result.
\begin{theorem}

\begin{enumerate}[(i)]
\item Let $L\in\cN$. Then there exists $b\in\C^*,z\in\C^*$ and $\lambda$ with $\lambda^2=q^{2d-3}\in q^{-3+2\N}$, such that $L\cong B_b^{(Y)}(N(\lambda,z))$. In particular, $L$ has a basis $\{v_{k,l}|k,l\in\Z,0\leq l\leq d-1\}$, on which the $\S_q$-action is given by
\begin{align*}
K.v_{k,l}=&\lambda b^{-1}q^{-k-2l}v_{k,l}, C.v_{k,l}=zv_{k,l},Y.v_{k,l}=v_{k+1,l}, \\
F.v_{k,l}=&\left\{\begin{array}{ll}v_{k,l+1}, & l<d-1,\\
-\sum\limits_{i=0}^{d-1}\Big(\frac{-q^{d-1}}{\lambda z(q+1)}\Big)^{d-i}\binom{d}{i}_qv_{k+2d-2i,i}, & l=d-1,
\end{array}\right.,\\
X.v_{k,l}=&-z\frac{bq^k-1}{q-1}v_{k-1,l}-\lambda^{-1}bq^{k+l-1}[l]_qv_{k+1,l-1},\\
E.v_{k,l}=&\frac{\lambda b^{-1}q^{1-k-l}-\lambda^{-1}bq^{k+l-1}}{q-q^{-1}}[l]_qv_{k,l-1} -\frac{q+q^2-b^{-1}q^{2-k}-bq^{k+1}}{(1-q^2)(q-1)}zv_{k-2,l}.
\end{align*}
\item Let $z\in\C^*$ and $\lambda^2=q^{2d-3}\in q^{-3+2\N}$. Then $B_b^{(Y)}(N(\lambda,z))\cong B_{b'}^{(Y)}(N(\lambda',z'))$ if and only if $\lambda=\lambda',z=z'$ and $b^{-1}b'\in q^\Z$.
\end{enumerate}
\end{theorem}

\begin{proof}
By Lemma \ref{hwv}, there exists $b^{-1}\in\C^*$ and $0\neq v\in B_{b^{-1}}^{(F)}(L)$ such that $Ev=0$. Then by Proposition \ref{hwsub}, $B_{b^{-1}}^{(F)}(L)$ contains a simple highest weight submodule $N$. However, $B_{b^{-1}}^{(F)}(L)$ is simple by Proposition \ref{twist}. So $L\cong B_b^{(F)}(N)$.

Following from Theorem \ref{simplehw}, one of the following holds:
\begin{enumerate}[(a)]
\item $XN=YN=0$ and $z=0$;
\item $N\cong M(\lambda,z)$ for some $\lambda$ with $\lambda^2\notin q^{-3+2\N}$ and $z\in\C^*$;
\item $N\cong N(\lambda,z)$ for some $\lambda$ with $\lambda^2\in q^{-3+2\N}$ and $z\in\C^*$.
\end{enumerate}

So, if $L$ has zero central charge, we have $XL=\Theta_b^{(F)}(X)N=0$ and $YL=\Theta_b^{(F)}(Y)N=0$ which is impossible.So $L$ has nonzero central charge. Then
using the same argument as above, we know that $L\cong B_b^{(Y)}(N)$ for some $b$ and some highest weight module $N$, where $N$ can only be the last two cases.
However, for case (b), the weight spaces of $B_b^{(Y)}(N)$ are unbounded. So, $N$ can be only case (c). From the basis of $B_b^{(Y)}(N(\lambda,z))$, we know that it is in $\cN$. The actions follows from direct calculations.

For (ii), suppose $B_b^{(Y)}(N(\lambda,z))\cong B_{b'}^{(Y)}(N(\lambda',z'))$. Then clearly they must have the same central charge must equal, that is $z=z'$. Also, the isomorphism implies the same dimension of weight spaces, which means $\lambda^2=\lambda'^2$. However, if $\lambda=-\lambda'$, then they are not isomorphic. From
\[
\lambda b^{-1}q^\Z=\mathrm{supp}(B_b^{(Y)}(N(\lambda,z)))=\mathrm{supp}(B_{b'}^{(Y)}(N(\lambda',z')))=\lambda' b'^{-1}q^\Z,
\]
we know that $b^{-1}b'\in q^\Z$.

Conversely, whenever $b=q^t\in q^\Z$, it is easy to check that
\[
B_1^{(Y)}(N(\lambda,z))\rightarrow B_b^{(Y)}(N(\lambda,z)); v_{k+t,l}\mapsto v_{k,l}
\]
is an isomorphism.
\end{proof}

\begin{center}
\bf Acknowledgments
\end{center}

\noindent The research presented in this paper was carried out during the visit of Y. Cai to Henan University. Y. Cai is partially supported by the China Postdoctoral Science Foundation (Grant 2016M600140)
 Y. Cheng. is partially supported by NSF of China (Grant
11047030) and the Science and Technology
Program of Henan Province (152300410061).
 G. L. is partially supported by NSF of China (Grant
11301143) and  the school fund of Henan University (yqpy20140044).

\vspace{1mm}
\noindent Y. Cai: Academy of Mathematics and Systems Science, Chinese Academy of
Sciences, Beijing, 100190, P.R. China. Email: yatsai@mail.ustc.edu.cn

\vspace{0.2cm}
\noindent Y. Cheng: School
of Mathematics and Statistics, Henan University, Kaifeng 475004, P.R. China. Email:
yscheng@henu.edu.cn

\vspace{0.2cm}
\noindent G. Liu: School
of Mathematics and Statistics, Henan University, Kaifeng 475004, P.R. China. Email:
liugenqiang@amss.ac.cn


\begin{thebibliography}{99}
\bibitem{B} V.V. Bavula, Simple $D[X, Y ; \sigma, a]$-modules, Ukrainian Math. J. 44(12), 1500-1511 (1992).
\bibitem{BL} V. V. Bavula and T. Lu, The prime spectrum and simple modules over the quantum spatial ageing algebra. Algebr. Represent. Theory 19 (2016), no. 5, 1109-1133.
\bibitem{CCS}Y. Cai, Y. Cheng, R. Shen, Quasi-Whittaker modules for the Schr\"{o}dinger algebra, Linear Algebra Appl. 463 16-32(2014) .


\bibitem{DD} V.K. Dobrev, H.-D. Doebner, C. Mrugalla, A $q$-Schr\"{o}dinger algebra, its lowest weight representations and generalized q-deformed heat/Schr\"{o}dinger equations, J. Phys. A 29  5909-5918(1996).
\bibitem{DD1} V. Dobrev, H.-D. Doebner, C. Mrugalla, Lowest weight representations of the Schr\"{o}dinger algebra and generalized heat/Schr\"{o}dinger equations, Rep. Math. Phys. 39  201-218(1997).
\bibitem{D} B. Dubsky, Classification of simple weight modules with finite-dimensional weight spaces over the Schr\"{o}dinger algebra, Linear Algebra Appl. 443  204-214(2014).
\bibitem{DLMZ} B. Dubsky, R. L¨¹, V. Mazorchuk and K. Zhao, Category O for the Schr\"{o}dinger algebra. Linear Algebra Appl. 460 (2014), 17-50.
\bibitem{GK} W. L. Gan, A. Khare, Quantized symplectic oscillator algebras of rank one, J. Algebra 310, no. 2, 671-707(2007).

\bibitem{L} G. Lusztig, Quantum groups at roots of 1, Geom. Dedicata 35 , no. 1-3, 89-113(1990).
\bibitem{LMZ} R. L¨¹, V. Mazorchuk, K. Zhao, On simple modules over conformal Galilei algebras, J. Pure Appl.
Algebra 218 1885-1899(2014).
\bibitem{M} O. Mathieu, Classification of irreducible weight modules, Ann. Inst. Fourier 50  537-592(2000).

\bibitem{WZ}Y. Wu and L. Zhu, Simple weight modules for Schr\"{o}dinger algebra, Linear Algebra Appl. 438, 559-563 (2013).

\bibitem{ZC} X. Zhang and Y. Cheng, Simple Schr\"{o}dinger modules which are locally finite over the positive part. J. Pure Appl. Algebra 219 (2015), no. 7, 2799-2815.
\end{thebibliography}
\end{document}